\documentclass{amsart}

\usepackage{tikz}
\usepackage{enumerate, verbatim}
\usepackage{hyperref}
\usepackage{cleveref}
\DeclareMathOperator{\sgn}{sgn}

\newtheorem{theorem}{Theorem}

\begin{document}

\title[A mass-action system with a vertical Andronov--Hopf bifurcation]{The smallest bimolecular mass-action system with a vertical Andronov--Hopf bifurcation}

\author[M. Banaji]{Murad Banaji}

\author[B. Boros]{Bal\'azs Boros}
\thanks{BB’s work was supported by the Austrian Science Fund (FWF), project P32532.}

\author[J. Hofbauer]{Josef Hofbauer}
\thanks{Email: m.banaji@mdx.ac.uk, balazs.boros@univie.ac.at, josef.hofbauer@univie.ac.at}


\date{}

\dedicatory{}

\begin{abstract}
We present a three-dimensional differential equation, which robustly displays a degenerate Andronov--Hopf bifurcation of infinite codimension, leading to a center, i.e., an invariant two-dimensional surface that is filled with periodic orbits surrounding an equilibrium. The system arises from a three-species bimolecular chemical reaction network consisting of four reactions. In fact, it is the only such mass-action system that admits a center via an Andronov--Hopf bifurcation.
\end{abstract}

\maketitle

\section{Summary of the main results} \label{sec:summary}

In order to admit an Andronov--Hopf bifurcation, the underlying chemical reaction network of a bimolecular mass-action system must have at least three species and at least four reactions. It has recently been shown that there are exactly $138$ nonisomorphic three-species four-reaction bimolecular reaction networks, whose associated mass-action systems admit Andronov--Hopf bifurcation \cite{banaji:boros:2022}. These networks fall into $87$ dynamically nonequivalent classes. Of these classes, $86$ admit nondegenerate Andronov--Hopf bifurcation for almost all parameter values on the bifurcation set, leading to isolated limit cycles. In the remaining class, however, the Andronov--Hopf bifurcation can only be degenerate. A representative of this exceptional class is 
\begin{align}\label{eq:ntw}
\begin{aligned}
\begin{tikzpicture}[scale=2]

\node (P1) at (0,0)    {$\mathsf{Z}+\mathsf{X}$};
\node (P2) at (1,0)    {$2\mathsf{X}$};
\node (P3) at (0,-0.4) {$\mathsf{X}+\mathsf{Y}$};
\node (P4) at (1,-0.4) {$2\mathsf{Y}$};
\node (P5) at (0,-0.8)   {$\mathsf{Y} + \mathsf{Z}$};
\node (P6) at (1,-0.8)   {$\mathsf{0}$};
\node (P7) at (2,-0.8)   {$2\mathsf{Z}$};

\draw[arrows={-stealth},thick] (P1) to node[above] {$\kappa_1$} (P2);
\draw[arrows={-stealth},thick] (P3) to node[above] {$\kappa_2$} (P4);
\draw[arrows={-stealth},thick] (P5) to node[above] {$\kappa_3$} (P6);
\draw[arrows={-stealth},thick] (P6) to node[above] {$\kappa_4$} (P7);


\end{tikzpicture}
\end{aligned}
\end{align}
giving rise to the mass-action differential equation
\begin{align}\label{eq:ode}
\begin{split}
\dot{x} &= x(\kappa_1 z - \kappa_2 y), \\
\dot{y} &= y(\kappa_2 x - \kappa_3 z), \\
\dot{z} &= z(-\kappa_3 y - \kappa_1 x) + 2\kappa_4
\end{split}
\end{align}
with state space ${\mathbb R}^3_{\geq 0}$,
where $\kappa_1$, $\kappa_2$, $\kappa_3$, $\kappa_4$ are positive parameters, called the reaction rate constants. (The other member of the exceptional class is obtained from \eqref{eq:ntw} by replacing the reaction $\mathsf{0} \to 2\mathsf{Z}$ by $\mathsf{0} \to \mathsf{Z}$.) The question left open in \cite{banaji:boros:2022} concerns the behaviour of system \eqref{eq:ode}. In \Cref{sec:analysis}, we prove that whenever the Jacobian matrix at the unique positive equilibrium has a pair of purely imaginary eigenvalues, the equilibrium is a center, i.e., there is a one parameter family of periodic orbits that fill the two-dimensional center manifold. In particular,  Andronov--Hopf bifurcations in system \eqref{eq:ode} are always \emph{vertical}, i.e., all the periodic orbits occur simultaneously at the critical value of the bifurcation parameter. Additionally, we prove that every positive solution converges either to one of these periodic orbits or to the unique positive equilibrium. Further, we show that the global center manifold is analytic and discuss how its closure intersects the boundary of the state space $\mathbb{R}^3_{\geq0}$.
\section{Vertical Andronov--Hopf bifurcations in mass--action systems}

There are two well-known small reaction networks that exhibit oscillations. The Lotka reactions \cite{lotka:1920} (left) and the Ivanova reactions \cite[page 630]{volpert:hudjaev:1985} (right) along with their associated mass-action differential equations are
\begin{center}
\begin{tikzpicture}[scale=1.8]

\node (P1) at (0,0)    {$\mathsf{X}$};
\node (P2) at (1,0)    {$2\mathsf{X}$};
\node (P3) at (0,-0.4) {$\mathsf{X}+\mathsf{Y}$};
\node (P4) at (1,-0.4) {$2\mathsf{Y}$};
\node (P5) at (0,-0.8)   {$\mathsf{Y}$};
\node (P6) at (1,-0.8)   {$\mathsf{0}$};

\draw[arrows={-stealth},thick] (P1) to node[above] {$\kappa_1$} (P2);
\draw[arrows={-stealth},thick] (P3) to node[above] {$\kappa_2$} (P4);
\draw[arrows={-stealth},thick] (P5) to node[above] {$\kappa_3$} (P6);

\node at (2,-0.4)
{$\begin{aligned}
\dot{x} &= x(\kappa_1 - \kappa_2 y) \\
\dot{y} &= y(\kappa_2 x - \kappa_3)
\end{aligned}$};

\begin{scope}[shift={(3.5,0)}]
\node (P1) at (0,0)    {$\mathsf{Z}+\mathsf{X}$};
\node (P2) at (1,0)    {$2\mathsf{X}$};
\node (P3) at (0,-0.4) {$\mathsf{X}+\mathsf{Y}$};
\node (P4) at (1,-0.4) {$2\mathsf{Y}$};
\node (P5) at (0,-0.8)   {$\mathsf{Y} + \mathsf{Z}$};
\node (P6) at (1,-0.8)   {$2\mathsf{Z}$};

\draw[arrows={-stealth},thick] (P1) to node[above] {$\kappa_1$} (P2);
\draw[arrows={-stealth},thick] (P3) to node[above] {$\kappa_2$} (P4);
\draw[arrows={-stealth},thick] (P5) to node[above] {$\kappa_3$} (P6);

\node at (2,-0.4)
{$\begin{aligned}
\dot{x} &= x(\kappa_1 z - \kappa_2 y) \\
\dot{y} &= y(\kappa_2 x - \kappa_3 z) \\
\dot{z} &= z(\kappa_3 y - \kappa_1 x)
\end{aligned}$};
\end{scope}

\draw[-] (3,-1.1) to (3,0.1);

\end{tikzpicture}
\end{center}
Both the Lotka and the Ivanova networks are \emph{bimolecular} (i.e., the molecularity of every reactant and product is at most two) and have rank two (i.e., the span of the vectors of the net changes of the species is two-dimensional). For the Lotka, the unique positive equilibrium is surrounded by periodic orbits, the level sets of $x^{\kappa_3}y^{\kappa_1}e^{-\kappa_2(x+y)}$. For the Ivanova, the triangle $\Delta_c=\{(x,y,z) \in \mathbb{R}_+^3 \colon x+y+z=c\}$ is invariant for any $c>0$, and the unique positive equilibrium in $\Delta_c$ is surrounded by periodic orbits, the level sets of $x^{\kappa_3}y^{\kappa_1}z^{\kappa_2}$. For both the Lotka and the Ivanova systems, the described behaviour holds for all $\kappa_1, \kappa_2, \kappa_3 >0$, and hence, these systems admit no bifurcation.

By \cite[Theorem 4.1]{boros:hofbauer:2022a}, the Lotka and the Ivanova systems are the only rank-two bimolecular mass-action systems with periodic orbits. Thus, for an Andronov--Hopf bifurcation to occur in a bimolecular mass-action system, its rank must be at least three, and hence, it must have at least three species. Moreover, by \cite[Lemma 2.3]{banaji:boros:2022}, it must have at least four reactions.

We turn to the question of when mass-action systems admit vertical Andronov--Hopf bifurcations. If we do not require bimolecularity then these can occur in rank-two networks. For example, by adding the reactions $\mathsf{X} \stackrel{\kappa_5} {\longleftarrow} 2 \mathsf{X} \stackrel{\kappa_4}{\longrightarrow} 3\mathsf{X}$ to the Lotka network above, the resulting mass-action system exhibits a vertical Andronov--Hopf bifurcation: for $\kappa_4$ slightly smaller than $\kappa_5$ the positive equilibrium is asymptotically stable, for $\kappa_4$ slightly larger than $\kappa_5$ it is repelling, while for $\kappa_4 = \kappa_5$ it is a center.

Focussing on bimolecular networks, we can construct rank-three networks with vertical Andronov--Hopf bifurcation. For instance, by inserting some intermediate steps into the Ivanova reactions and choosing the rate constants appropriately, we obtain the following rank-three bimolecular mass-action system with cyclic symmetry:
\begin{align}\label{eq:9reactions}
\begin{aligned}
\begin{tikzpicture}[scale=2]

\node (P1) at (0,0)    {$\mathsf{Z}+\mathsf{X}$};
\node (P2) at (1,0)    {$\mathsf{X}$};
\node (P3) at (2,0)    {$2\mathsf{X}$};
\node (P4) at (0,-1/2) {$\mathsf{X}+\mathsf{Y}$};
\node (P5) at (1,-1/2) {$\mathsf{Y}$};
\node (P6) at (2,-1/2) {$2\mathsf{Y}$};
\node (P7) at (0,-1)   {$\mathsf{Y} + \mathsf{Z}$};
\node (P8) at (1,-1)   {$\mathsf{Z}$};
\node (P9) at (2,-1)   {$2\mathsf{Z}$};

\draw[arrows={-stealth},thick] (P1) to node[above] {$\alpha$} (P2);
\draw[arrows={-stealth},thick,transform canvas={yshift=2pt}] (P2) to node[above] {$\gamma$} (P3);
\draw[arrows={-stealth},thick,transform canvas={yshift=-2pt}] (P3) to node[below] {$\beta$} (P2);

\draw[arrows={-stealth},thick] (P4) to node[above] {$\alpha$} (P5);
\draw[arrows={-stealth},thick,transform canvas={yshift=2pt}] (P5) to node[above] {$\gamma$} (P6);
\draw[arrows={-stealth},thick,transform canvas={yshift=-2pt}] (P6) to node[below] {$\beta$} (P5);

\draw[arrows={-stealth},thick] (P7) to node[above] {$\alpha$} (P8);
\draw[arrows={-stealth},thick,transform canvas={yshift=2pt}] (P8) to node[above] {$\gamma$} (P9);
\draw[arrows={-stealth},thick,transform canvas={yshift=-2pt}] (P9) to node[below] {$\beta$} (P8);

\node at (3.5,-1/2) {$\begin{aligned} \dot{x} & = x(\gamma-\beta x - \alpha y) \\
\dot{y} & = y(\gamma-\beta y - \alpha z) \\
\dot{z} & = z(\gamma-\beta z - \alpha x) 
\end{aligned}$};

\end{tikzpicture}
\end{aligned}
\end{align}
This system exhibits vertical Andronov--Hopf bifurcation: the unique positive equilibrium is asymptotically stable for $\alpha<2\beta$, it is unstable for $\alpha>2\beta$, while it is a center for $\alpha=2\beta$. More precisely, for $\alpha=2\beta$ the triangle $\Delta=\{(x,y,z)\in\mathbb{R}^3_+\colon x+y+z=\frac{\gamma}{\beta}\}$ is invariant, and on $\Delta$ the equilibrium $(x^*,y^*,z^*)=\left(\frac{\gamma}{3\beta},\frac{\gamma}{3\beta},\frac{\gamma}{3\beta}\right)$ is surrounded by periodic orbits. On these curves, the function $xyz$ is constant, as in the Ivanova system with equal rate constants. In fact, the function $\frac{xyz}{(x+y+z)^3}$ is a constant of motion in $\mathbb{R}^3_+$, the stable manifold of $(x^*,y^*,z^*)$ is the line $x=y=z$ in $\mathbb{R}^3_+$, while the $\omega$-limit set of any positive initial point outside this line is one of the periodic orbits in $\Delta$, see \cite{may:leonard:1975} or \cite[Section 5.5]{hofbauer:sigmund:1998}.

In the next section, we prove that the mass-action system \eqref{eq:ode} (which is obtained from the Ivanova network by adding a single intermediate step, and has only four reactions) also admits a vertical Andronov--Hopf bifurcation, even though it has no obvious symmetries. By \cite[Theorems 5.2 and 7.1]{banaji:boros:2022}, it is the \emph{only} three-species four-reaction bimolecular mass-action system that exhibits a vertical Andronov--Hopf bifurcation.
\section{Analysis}
\label{sec:analysis}

In this section we analyse the mass-action system \eqref{eq:ode}. In particular, we show that it undergoes a vertical Andronov--Hopf bifurcation at $\kappa_1 = \kappa_2 + \kappa_3$. A description of the dynamics in the critical case is provided in \Cref{thm:main}, while some information on the shape of the global center manifold is revealed in \Cref{thm:intersect_boundary}.

By a short calculation, system \eqref{eq:ode} has a unique positive equilibrium, given by
\begin{align*}
(x^*,y^*,z^*)=\left(\sqrt{\frac{\kappa_3 \kappa_4}{\kappa_1\kappa_2}}, \sqrt{\frac{\kappa_1 \kappa_4}{\kappa_2\kappa_3}}, \sqrt{\frac{\kappa_2 \kappa_4}{\kappa_1\kappa_3}}\right).
\end{align*}
Denoting by $J$ the Jacobian matrix at $(x^*,y^*,z^*)$, one finds that the characteristic polynomial of $J$ equals $\lambda^3 + a_2 \lambda^2 + a_1 \lambda + a_0$ with
\begin{align*}
a_2 = 2\sqrt{\frac{\kappa_1 \kappa_3 \kappa_4}{\kappa_2}},\quad
a_1 = (\kappa_1 + \kappa_2 - \kappa_3)\kappa_4,\quad
a_0 = 4\sqrt{\kappa_1 \kappa_2 \kappa_3 \kappa_4^3}.
\end{align*}
Since $a_0>0$, one eigenvalue is a negative real number. Further, observe that $a_2>0$ and $\sgn(a_2a_1-a_0) = \sgn(\kappa_1 - \kappa_2 - \kappa_3)$. Therefore, by the Routh--Hurwitz criterion,
\begin{enumerate}[(a)]
\item if $\kappa_1 > \kappa_2 + \kappa_3$ then all three eigenvalues of $J$ have negative real parts (and thus, the positive equilibrium is asymptotically stable),
\item if $\kappa_1 = \kappa_2 + \kappa_3$ then $J$ has a pair of purely imaginary eigenvalues,
\item if $\kappa_1 < \kappa_2 + \kappa_3$ then $J$ has  eigenvalues with positive real parts (and thus, the positive equilibrium is unstable).
\end{enumerate}
Thus, on the bifurcation set (given by $\kappa_1 = \kappa_2 + \kappa_3$), apart from the negative real eigenvalue, there is a pair of imaginary eigenvalues ($\pm \omega i$ with $\omega=\sqrt{2\kappa_2\kappa_4}$). It was shown by direct calculation in \cite{banaji:boros:2022} that the first and the second focal values (also known as Poincar\'e--Lyapunov coefficients \cite{farkas:1994}, or Lyapunov coefficients \cite{kuznetsov:2004}) both vanish.  In the sequel, we show by providing a constant of motion that the system~\eqref{eq:ode} has a center whenever $\kappa_1 = \kappa_2 + \kappa_3$. Therefore, by a theorem of Lyapunov (see \cite[page 143 and Theorem 7.2.1]{farkas:1994}), in fact, the $k$th focal value vanishes for all $k\geq1$, and system \eqref{eq:ode} exhibits a vertical Andronov--Hopf bifurcation as $\kappa_1$ varies through  $\kappa_2+\kappa_3$.

\begin{theorem} \label{thm:main}
For the mass-action system \eqref{eq:ode} with $\kappa_1 = \kappa_2 + \kappa_3$, the following hold.
\begin{enumerate}[(i)]
\item
The function 
\begin{align*}
V(x,y,z) = \frac{\kappa_2}{2} (x-y+z)^2 + 2 \kappa_2 xy - 2 \kappa_4 \log (xy)
\end{align*}
is a constant of motion.
\item The stable manifold of $(x^*,y^*,z^*)$ is $\{(x,y,z)\in\mathbb{R}^3_{\geq 0}\colon x-y+z=0, xy=\frac{\kappa_4}{\kappa_2}\}$.
\item There exists an analytic two-dimensional invariant surface $\mathcal{M}$ in $\mathbb{R}^3_+$, composed of periodic orbits and the positive equilibrium. The $\omega$-limit set of any positive initial condition is either the positive equilibrium or one of the periodic orbits in $\mathcal{M}$.
\end{enumerate}
\end{theorem}
\begin{proof}
We perform a change of coordinates which reveals the global orbit structure of \eqref{eq:ode}. The map 
\begin{align} \label{eq:Psi}
\Psi \colon \begin{pmatrix} x \\ y \\ z \end{pmatrix}
\mapsto
\begin{pmatrix} x-y+z\\ \log x+\log y -\log \frac{\kappa_4}{\kappa_2}\\ -\log x + \log\sqrt{\frac{\kappa_3\kappa_4}{\kappa_1\kappa_2}} \end{pmatrix}
\end{align}
is an analytic diffeomorphism between $\{(x,y,z) \in \mathbb{R}^3\colon x >0, y>0\}$ and $\mathbb{R}^3$. Its inverse is given by
\begin{align*}
\Psi^{-1} \colon \begin{pmatrix} p \\ q \\ r \end{pmatrix}
\mapsto
\begin{pmatrix} \frac{\alpha}{\kappa_1} e^{-r} \\ \frac{\alpha}{\kappa_3} e^{q+r} \\ p - \frac{\alpha}{\kappa_1} e^{-r} + \frac{\alpha}{\kappa_3} e^{q+r} \end{pmatrix},
\end{align*}
where $\alpha = \sqrt{\frac{\kappa_1 \kappa_3 \kappa_4}{\kappa_2}}$. When $\kappa_1 = \kappa_2 + \kappa_3$, the new coordinates $(p,q,r) = \Psi(x,y,z)$ evolve according to the differential equation
\begin{align}
\label{eq:ode_pqr}
\begin{split}
\dot p & = -2\kappa_4(e^q-1), \\
\dot q & = \kappa_2 p, \\
\dot r& = -\kappa_1 p+\alpha(e^{-r}-e^{q+r}).
\end{split}
\end{align}

Notice that $p$ and $q$ evolve independently of $r$. In fact, the $(p,q)$-system is Newtonian (i.e., $\ddot{q}+F(q)=0$), and thus, it is also Hamiltonian. Its Hamiltonian function is
\begin{align*}
H(p,q) = \frac{\kappa_2}{2} p^2+2\kappa_4(e^q-q - 1).
\end{align*}
Since $H\circ (\Psi_1, \Psi_2)$ differs from the function $V$ in (i) only by an additive constant, $V$ is indeed a constant of motion in the original coordinates, proving (i). Observe furthermore that the $r$-axis is invariant for \eqref{eq:ode_pqr}, and the flow there converges to the origin. Thus, the $r$-axis is the stable manifold, and in turn, this shows (ii).

The level sets of $H$ are closed, bounded curves which foliate the $(p,q)$-plane. Thus, in the $(p,q)$-system, the origin is a global center, i.e., each nonconstant solution is a periodic one whose orbit surrounds the origin, see \Cref{fig:hamiltonian} for a phase portrait.
\begin{figure}
    \centering
    \includegraphics[scale=0.55]{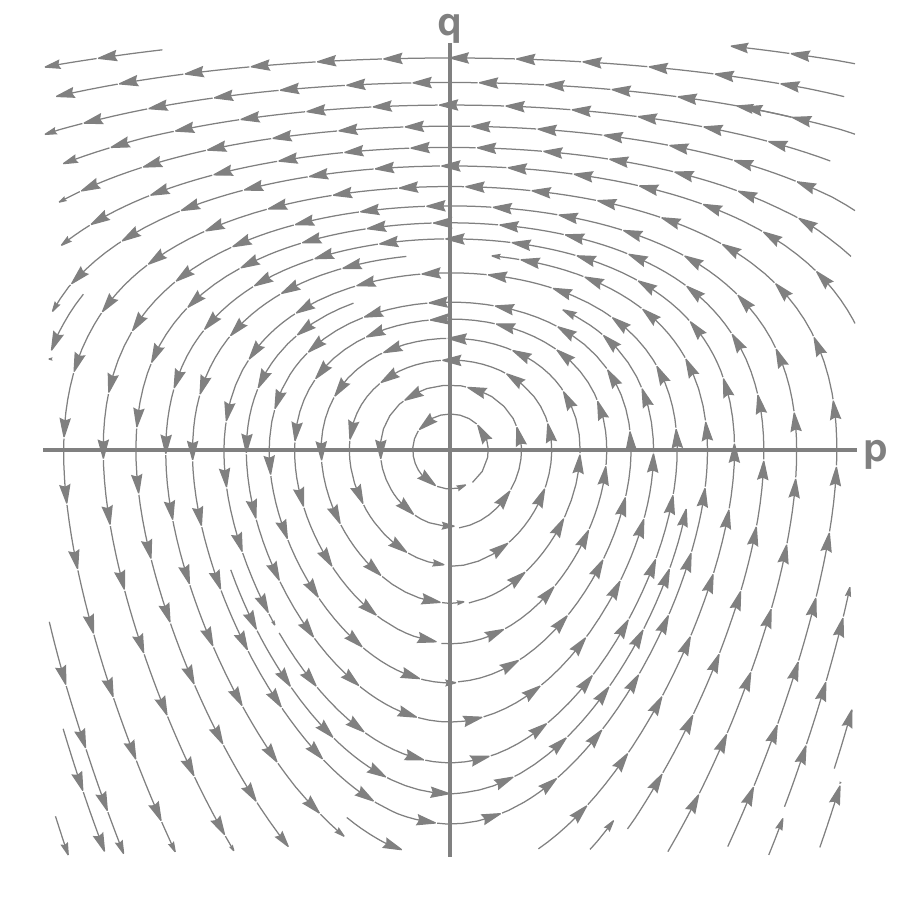}
    \caption{The phase portrait of $\dot p = -2\kappa_4(e^q-1)$, $\dot q = \kappa_2 p$. All nonconstant solutions are periodic.}
    \label{fig:hamiltonian}
\end{figure}
The function $H$ also provides an analytic constant of motion for system~\eqref{eq:ode_pqr}. Thus, the Lyapunov Center Theorem (see e.g., \cite{kelley:1969}, \cite[Theorem 3]{edneral:mahdi:romanovski:shafer:2012}, or \cite[Theorem 5.1.1]{romanovski:shafer:2016})   shows that the local center manifold at the equilibrium $(0,0,0)$ is unique, analytic, and filled with periodic orbits. In the following we show that this center manifold extends globally and attracts every solution.

For any $L > 0$, the cylinder $C_L = \{(p,q,r)\in\mathbb{R}^3 \colon H(p,q) = L\}$ is invariant for the differential equation \eqref{eq:ode_pqr}, and there exist $\underline{r} < 0 < \overline{r}$ such that
\begin{align*}
\text{$\dot{r}<0$ in $\{(p,q,r)\in C_L \colon r>\overline{r}\}$ and
$\dot{r}>0$ in $\{(p,q,r)\in C_L \colon r<\underline{r}\}$}
\end{align*}
hold. Therefore, the bounded cylinder $\{(p,q,r)\in C_L \colon \underline{r} \leq r\leq \overline{r}\}$ is forward invariant, and attracts all orbits on $C_L$. This shows, in particular, that all solutions of the differential equation \eqref{eq:ode_pqr} exist for all positive time and so the differential equation \eqref{eq:ode_pqr} defines a semiflow $\Phi^t$ on $\mathbb{R}^3$. On the other hand, the $(p,q)$ subsystem is associated with a flow $\widehat{\Phi}^t$ on the $(p,q)$-plane since all orbits are bounded and thus exist for all time. Both $\Phi^t$ and $\widehat{\Phi}^t$ are analytic (by the analytic dependence of solutions on initial conditions).

Next, we show that any two solutions starting above each other on a cylinder $C_L$ approach each other. Let us denote the r.h.s.\ of \eqref{eq:ode_pqr} as $f(p,q,r)$, and accordingly, $f_3(p,q,r)$ equals $-\kappa_1 p + \alpha(e^{-r}-e^{q+r})$. Note that
\begin{align} \label{eq:partialbound}
\frac{\partial f_3}{\partial r} = -\alpha(e^{-r}+e^{q+r}) \leq -2\alpha e^{q/2}.
\end{align}
For a fixed $L>0$, let $(p,q,r_1), (p,q,r_2) \in C_L$ with $r_1 < r_2$. Further, let $r_i(t) = \Phi^t_3(p,q, r_i)$, the third component of the solution. Then $r_1(t) < r_2(t)$ for all $t>0$ and, by the Mean Value Theorem, 
\begin{align*}
\dot  r_2(t) - \dot  r_1(t) &= 
 f_3(p(t), q(t),  r_2(t)) -  f_3(p(t), q(t), r_1(t)) =\\
  &= (r_2(t) - r_1(t)) \frac{\partial f_3}{\partial r}(p(t),q(t),\widetilde r(t))
\end{align*}
with $r_1(t) \leq \widetilde{r}(t) \leq r_2(t)$. By \eqref{eq:partialbound}, it follows that
\begin{align*}
\dot{r}_2(t)-\dot{r}_1(t) \leq -K (r_2(t) - r_1(t))
\end{align*}
holds with $K = 2\alpha e^{\overline{q}/2}$, where $\overline{q}$ is the negative solution of $H(0,q)=L$. Thus, by the Gronwall Lemma,
\begin{align}\label{eq:contraction_flow}
|r_2(t) - r_1(t)| \leq e^{-K t} |r_2 - r_1|.
\end{align}

Next, we define the Poincar\'e section
\begin{align*}
\Sigma = \{(p,q,r) \in \mathbb{R}^3\colon p=0, q > 0\},
\end{align*}
and a Poincar\'e map $P\colon \Sigma \to \Sigma$ as follows. For any $q> 0$, let $\ell_q$ be the line $\{0\} \times \{q\} \times \mathbb{R}$. Then $(\ell_q)_{q>0}$ is a foliation of $\Sigma$. Associated with each $q>0$ is a minimal positive period $\tau_q$ such that $\widehat{\Phi}^{\tau_q}(0,q) = (0,q)$, i.e., $\Phi^{\tau_q}(\ell_q) \subseteq \ell_q$. By the analytic Implicit Function Theorem, $\tau_q$ is an analytic function of $q$. We can thus define the first return map $P$ by $P(0,q,r) = \Phi^{\tau_q}(0,q,r)$, and since $\Phi$ and $\tau_q$ are analytic, $P$ is analytic on $\Sigma$.

We define the analytic function $R\colon (0,\infty) \times \mathbb{R} \to \mathbb{R}$ by  $P(0,q,r) = (0,q,R(q,r))$. For any fixed $q>0$, by substituting $t=\tau_q$ into \eqref{eq:contraction_flow}, we obtain
\begin{align}\label{eq:contraction_R}
|R(q,r_2)-R(q,r_1)| \leq e^{-K\tau_q}|r_2-r_1|,
\end{align}
showing that $R(q,\cdot)$ is a contraction. Hence, for each $q > 0$ the function $R(q,\cdot)\colon \mathbb{R} \to \mathbb{R}$ has a unique fixed point $h(q)$. Every orbit of $P$ starting on the line $\ell_q$ converges to $(0,q,h(q))$ which corresponds to a periodic orbit of \eqref{eq:ode_pqr} with period $\tau_q$. Additionally, since $\left|\frac{\partial R}{\partial r}\right|\leq e^{-K\tau_q}<1$ follows from \eqref{eq:contraction_R}, the analytic Implicit Function Theorem applies to $R(q,h(q)) = h(q)$, and thus, $h$ is analytic for $q>0$.

Finally, applying $\Phi^t$ to the graph of $h$, we obtain the invariant surface $\mathcal{C}=\{\Phi^t(0,q,h(q)): q > 0, t \in \mathbb{R}\}\cup\{(0,0,0)\}$, consisting entirely of periodic orbits of the flow (together with the equilibrium).  Near the origin, $\mathcal{C}$ coincides with the local center manifold, hence, $\mathcal{C}$ is analytic there by the Lyapunov Center Theorem. That $\mathcal{C}$ is analytic away from the origin follows by a straightforward argument that uses the analyticity of $\Phi$, $\widehat{\Phi}$, and $h$.

Setting $\mathcal{M} = \Psi^{-1}(\mathcal{C})$ and recalling that $\Psi$ is an analytic diffeomorphism complete the proof of statement (iii).
\end{proof}

In the next theorem, we describe how the closure of the surface $\mathcal{M}$ intersects the boundary of the nonnegative orthant $\mathbb{R}^3_{\geq0}$. We call a solution $t \mapsto (x(t),y(t),z(t))$ \emph{complete} if it is defined for all $t \in \mathbb{R}$.

\begin{theorem} \label{thm:intersect_boundary}
For the invariant surface $\mathcal{M}$, obtained in \Cref{thm:main}, the intersection $\overline{\mathcal{M}} \cap \partial\mathbb{R}^3_{\geq0}$ is the parametric curve
\begin{align*}
\sqrt{\frac{2\kappa_4}{\kappa_3}}
\left(0,\frac{\varphi(\tau)}{\Phi(\tau)}, \frac{\varphi(\tau)}{\Phi(\tau)} + \tau\right) \text{ for } \tau \in \mathbb{R},
\end{align*}
where $\varphi(\tau) = \frac{1}{\sqrt{2\pi}}e^{-\frac{\tau^2}{2}}$ and $\Phi(\tau) = \int_{-\infty}^\tau \varphi(s) \mathrm{d}s$. Up to the rescaling $\tau = \sqrt{2\kappa_3\kappa_4}t$ of time, this is the only complete solution of \eqref{eq:ode} on the boundary of $\mathbb{R}^3_{\geq0}$.
\end{theorem}
\begin{proof}
First, observe that $\frac{\mathrm{d}}{\mathrm{d}t} (yz)< 0$ whenever $yz \geq\frac{2\kappa_4}{\kappa_3}$. Indeed, 
\begin{align*}
\frac{\mathrm{d}}{\mathrm{d}t} \log(yz) =  \frac{\dot y}{y} + \frac{\dot z}{z} &= \kappa_2 x - \kappa_3 z - \kappa_3 y - \kappa_1 x + \frac{2 \kappa_4}{z} = \\
&= - \kappa_3 (x+z) + \frac{2 \kappa_4- \kappa_3 yz}{z}  < 0,
\end{align*}
where we used $\kappa_1 = \kappa_2 + \kappa_3$ and $yz\geq\frac{2\kappa_4}{\kappa_3}$. As a consequence, for any point $(x,y,z) \in \mathcal{M}$ we have $yz\leq \frac{2\kappa_4}{\kappa_3}$.

Next, we show that $\overline{\mathcal{M}}$ intersects the facet $\mathcal{F}=\{(x,y,z) \in \mathbb{R}^3_{\geq0} \colon x=0\}$. To this end, take a sequence of points $(p_n,q_n,r_n)_{n\geq0}\subseteq\mathcal{C}$ such that $p_n = -1$ and $\lim_{n\to\infty}q_n =-\infty$, where $\mathcal{C}$ is the invariant surface of the differential equation \eqref{eq:ode_pqr}, constructed in the proof of \Cref{thm:main}, foliated by periodic orbits. Then define $(x_n,y_n,z_n) = \Psi^{-1}(p_n,q_n,r_n) \in \mathcal{M}$, where $\Psi$ is given by \eqref{eq:Psi}. Since $p_n=-1$, it follows that $y_n = x_n + z_n + 1$, and consequently, $y_n\geq 1$. Since $\lim_{n\to\infty}q_n =-\infty$, we obtain that $\lim_{n\to\infty} x_n y_n = 0$. Hence, $\lim_{n\to\infty} x_n = 0$ and $\lim_{n\to\infty} (y_n - z_n) = 1$. Taking also into account that $(x,y,z) \in \mathcal{M}$ implies $yz \leq \frac{2\kappa_4}{\kappa_3}$, the sequence $(x_n,y_n,z_n)_{n\geq0}$ has an accumulation point on the line segment $\{(0,y,z)\in\mathbb{R}^3_{\geq0} \colon y-z=1 \text{ and }yz\leq \frac{2\kappa_4}{\kappa_3}\}$.

Since $\mathcal{M}$ consists of orbits of complete solutions, so does the closure of $\mathcal{M}$. Therefore, since $\overline{\mathcal{M}} \subseteq \mathbb{R}^3_{\geq0}$, there is a complete solution in $\mathbb{R}^3_{\geq0}$ through the accumulation point that we found in the previous paragraph. Since the set $\mathcal{G}_1 = \{(x,y,z) \in \mathbb{R}^3 \colon x=0, y\geq0\}$ is invariant, this complete solution lies in $\mathcal{G}_1 \cap \mathbb{R}^3_{\geq0}$, i.e., in $\mathcal{F}$.

Next, we investigate the dynamics on $\mathcal{G}_1$. To ease the notation, we divide both $y$ and $z$ by $\sqrt{\frac{2\kappa_4}{\kappa_3}}$. After also rescaling time ($\tau = \sqrt{2\kappa_3\kappa_4}t$), the differential equation \eqref{eq:ode} on $\mathcal{G}_1$ becomes
\begin{align}\label{eq:ode_yz_facet}
\begin{split}
\dot{y} &= -yz, \\
\dot{z} &= -yz +1.
\end{split}
\end{align}
The general solution to \eqref{eq:ode_yz_facet}, up to time shift, is
\begin{align}\label{eq:facet_solution}
\begin{split}
y(\tau) &= \frac{\varphi(\tau)}{\Phi(\tau)+C}, \\
z(\tau) &= \frac{\varphi(\tau)}{\Phi(\tau)+C} + \tau,
\end{split}
\end{align}
where $-1 < C \leq \infty$ (the limit case $C = \infty$ gives the complete solution $y(\tau) = 0$, $z(\tau) = \tau$ along the $z$-axis). For $-1<C<0$ the solution \eqref{eq:facet_solution} is defined only in the interval $(\tau_0, \infty)$, where $\tau_0$ is given by $\Phi(\tau_0)+C=0$, and thus, the solution is not complete.  For $C>0$, the solution \eqref{eq:facet_solution} is defined for all $\tau\in\mathbb{R}$, however, since $\lim_{\tau\to -\infty}z(\tau)=-\infty$, it is a complete solution in $\mathcal{G}_1$, but not in $\mathcal{F}$. Consequently, the only complete solution in $\mathcal{F}$ is \eqref{eq:facet_solution} with $C=0$. See \Cref{fig:yz_plane} for the orbits of the solutions \eqref{eq:facet_solution} for different values of $C$.

On the invariant set $\mathcal{G}_2=\{(x,y,z) \in \mathbb{R}^3 \colon x\geq0, y=0\}$, the differential equation~\eqref{eq:ode} takes the form
\begin{align*}
\dot{x} &= \kappa_1 xz, \\
\dot{z} &= -\kappa_1 xz +2\kappa_4.
\end{align*}
Since $x(t)+z(t)=2\kappa_4 t+ C$ (for some $C\in\mathbb{R}$), for every solution with $(x(0),z(0)) \in \mathbb{R}^2_{\geq0}$ there exists a time $t^*<0$ such that $z(t^*)<0$. Thus, there is no complete solution in $\mathcal{G}_2 \cap \mathbb{R}^3_{\geq0}$.

Finally, since $\dot{z}>0$ for \eqref{eq:ode} whenever $z=0$, the closure of $\mathcal{M}$ cannot intersect $\{(x,y,z) \in \mathbb{R}^3_{\geq0} \colon z=0\}$. This concludes the proof of the theorem. (The shape of the global center manifold $\mathcal{M}$ is shown in \Cref{fig:M}.)
\end{proof}

\begin{figure}[b]
    \centering
    \includegraphics[scale=0.55]{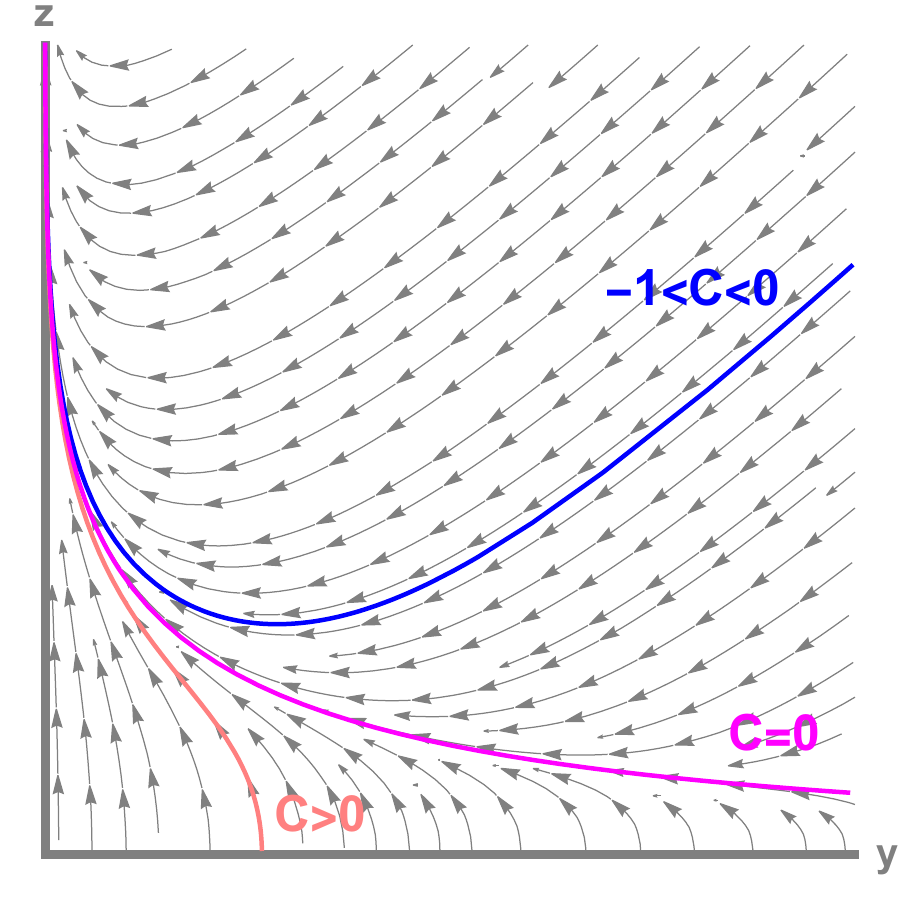}
    \caption{The phase portrait of the differential equation \eqref{eq:ode_yz_facet}, along with three highlighted trajectories (for $-1<C<0$, $C=0$, $C>0$). The trajectory shown in magenta is the only one that corresponds to a complete solution that lies entirely in the boundary of the nonnegative orthant $\mathbb{R}^3_{\geq0}$.}
    \label{fig:yz_plane}
\end{figure}

\begin{figure}[b]
    \centering
    \includegraphics[scale=0.28]{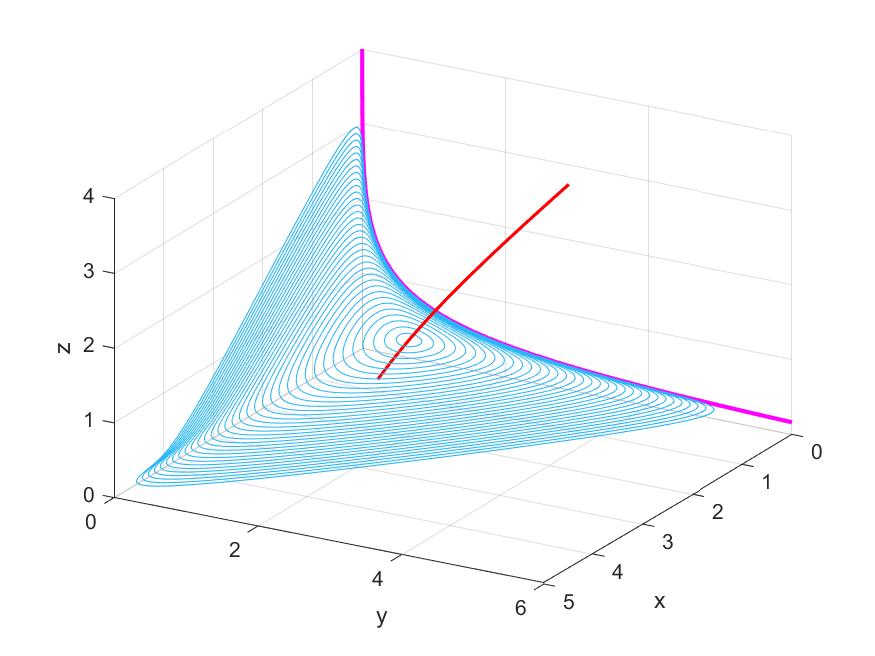}
    \caption{The periodic orbits of the differential equation \eqref{eq:ode} (shown in blue), the stable manifold of the unique positive equilibrium (shown in red), and the intersection of the closure of the center manifold with the boundary of the positive orthant (shown in magenta).}
    \label{fig:M}
\end{figure}
\clearpage
\section{Discussion}

We have shown in \Cref{thm:main} that the positive equilibrium of \eqref{eq:ode} is a center when $\kappa_1 = \kappa_2 + \kappa_3$. In fact, we provided a constant of motion $V$, and proved the existence of a global center manifold $\mathcal{M}$ that attracts all positive solutions, albeit we have no explicit formula for $\mathcal{M}$. The periodic orbits are obtained as the intersection of the level sets of $V$ with $\mathcal{M}$. On the other hand, a frequent situation in the literature is when the center manifold $\mathcal{M}$ is known explicitly, but the function $V$ is not (although its restriction to $\mathcal{M}$ may be known). In some cases (e.g.\ for system~\eqref{eq:9reactions}), both $V$ and $\mathcal{M}$ are known explicitly. For some examples of centers on center manifolds, see e.g.\ \cite{edneral:mahdi:romanovski:shafer:2012}, \cite{garcia:maza:shafer:2019}, or \cite[Section 5.2]{romanovski:shafer:2016}.

As was discussed in \Cref{sec:summary}, there are $86$ dynamically nonequivalent three-species four-reaction bimolecular mass-action systems that admit a nondegenerate Andronov--Hopf bifurcation. Of those, $31$ also admit a degenerate Andronov--Hopf bifurcation (i.e., a vanishing first focal value) on an exceptional subset of the bifurcation set, see \cite{banaji:boros:2022}. However, in all $31$ cases, the second focal value is nonzero on this exceptional set, and thus, degenerate Andronov--Hopf bifurcations of codimension greater than two are impossible. Thus, system \eqref{eq:ode} stands out in two ways: the Andronov--Hopf bifurcation is degenerate everywhere on the bifurcation set; and additionally all focal values vanish, leading to a center through a  bifurcation of infinite codimension.

We conclude with two open questions about system \eqref{eq:ode}:
\begin{enumerate}[(a)]
\item For $\kappa_1 > \kappa_2 + \kappa_3$, is the positive equilibrium \emph{globally asymptotically stable}?
\item For $\kappa_1 < \kappa_2 + \kappa_3$, are all solutions outside the stable manifold of the positive equilibrium \emph{unbounded}?
\end{enumerate}

\bibliographystyle{abbrv}
\bibliography{biblio}

\end{document}